\documentclass[11pt,reqno]{amsart}

\usepackage{arxiv}
\usepackage{amsmath,amssymb,amsthm,mathtools}
\usepackage{tikz}
\numberwithin{equation}{section}
\renewcommand{\E}{\mathbb{E}}
\renewcommand{\Pr}{\mathbb{P}}

\newcommand{\cW}{\mathcal{W}}
\newcommand{\cM}{\mathcal{M}}
\newcommand{\cL}{\mathcal{L}}
\newcommand{\cH}{\mathcal{H}}
\newcommand{\cP}{\mathcal{P}}
\newcommand{\cK}{\mathcal{K}}
\newcommand{\cG}{\mathcal{G}}
\newcommand{\cE}{\mathcal{E}}
\newcommand{\Ent}{\mathcal{I}}
\newcommand{\dens}{\mathsf{d}}
\newcommand{\dist}{\operatorname{dist}}
\newcommand{\KL}{D_{\mathrm{KL}}}
\renewcommand{\dd}{\,\mathrm{d}}
\renewcommand{\ip}{i_p}
\newcommand{\cut}{\square}

\DeclareMathOperator{\Ber}{Ber}

\hypersetup{
  pdftitle={A first order phase transition for lower tails in sparse random graphs},
  pdfauthor={Byron Chin},
  pdfkeywords={subgraph lower tail, graphon variational problem, first order phase transition, replica symmetry, relative entropy}
}

\title[First order transition for sparse lower tails]
{A first order phase transition for lower tails in sparse random graphs}

\author{Byron Chin}
\address{School of Mathematics, Georgia Institute of Technology}
\email{bchin@gatech.edu}

\begin{document}

\begin{abstract}
We study the graphon variational problem that governs lower tail large deviations of the number of copies of a fixed nonbipartite graph $H$ in a sparse Erd\H{o}s--R\'enyi random graph. Our main new insight is that any non-constant minimizer of the variational problem exhibits a discontinuous jump away from the constant graphon. We deduce that the replica-symmetric and symmetry-breaking regimes are separated by a single critical parameter and that the phase transition is first order. In the supercritical regime $p \gg n^{-1/m_2(H)}$, we transfer these results to the conditioned random graph, showing that the edge density is a discontinuous order parameter and that the conditioned graph concentrates in cut distance near the full minimizer set.
\end{abstract}

\maketitle

\section{Introduction}

Let $H$ be a fixed graph and let $G(n,p)$ be the Erd\H{o}s--R\'enyi random graph. We write $N_H(G)$ for the number of copies of $H$ in a graph $G$ and $N_H(p):=N_H(G(n,p))$, so that $\E[N_H(p)]=N_Hp^{e(H)}$, where $N_H = N_H(1)$ is the number of copies of $H$ in $K_n$. For $\eta\in(0,1)$ we study the lower tail large deviation event
\begin{equation}\label{eq:LT}
  \cL_p(H,\eta):=\bigl\{G:N_H(G)\le\eta\,\E[N_H(p)]\bigr\},
\end{equation}
and we parametrize the deficit by $q\in(0,1)$ through $\eta=q^{e(H)}$.

The graphon variational approach to study such events was initiated by the seminal work of Chatterjee and Varadhan \cite{ChatterjeeVaradhan}, who established a large deviation principle for $G(n,p)$ at constant $p$ in the space of graphons. Their result implies two statements about the lower tail event. The first is that the exponential rate of $\Pr(\cL_p(H,\eta))$ is the value of an entropy variational problem over graphons; the second is that a graph drawn from the conditional distribution is exponentially unlikely to lie far in cut distance from the set of optimizers of that problem. While graphons are no longer a convenient limit object for sparse random graphs where $p=o(1)$, the variational answer to the lower tail problem continues to hold. Kozma and Samotij \cite{KozmaSamotij} computed the lower tail asymptotics for an arbitrary graph $H$ in terms of the variational problem throughout the range $p\gg n^{-1/m_2(H)}$, which is the optimal range in which such a result can be expected to hold. This ``supercritical'' regime \cite{JenssenPerkinsPotukuchiSimkin} is the range of sparsities we consider as well.

The variational problem that arises in sparse graphs is the following. Let
\begin{equation*}
  h(x):=x\log x-x+1,\qquad x\in[0,1],
\end{equation*}
with $h(0)=1$, so that $h$ is the $p\to0$ limit of the rescaled relative entropy, in the sense that $p^{-1}\ip(px)\to h(x)$ uniformly on $[0,1]$, where $\ip(x)$ denotes the relative entropy of $\Ber(x)$ with respect to $\Ber(p)$. For a graphon $W\colon[0,1]^2\to[0,1]$ set
\begin{equation*}
  \Ent(W):=\int_{[0,1]^2}h(W),\qquad
  \dens(W):=\int_{[0,1]^2}W,\qquad
  t(H,W):=\int_{[0,1]^{V(H)}}\prod_{ij\in E(H)}W_{x_ix_j},
\end{equation*}
where we abbreviate $W_{xy}=W(x,y)$ and suppress Lebesgue measure. We call $\Ent$ the \emph{entropy cost}, $\dens$ the \emph{edge density}, and $t(H,\cdot)$ the \emph{$H$-density}. The lower tail variational problem for $H$ is then
\begin{equation}\label{eq:Phi}
  \Phi(H,\eta):=\min\bigl\{\Ent(W):t(H,W)\le\eta\bigr\},
\end{equation}
where the minimum is attained by compactness of the graphon space, and we abbreviate $\Phi(q):=\Phi(H,q^{e(H)})$. The constant graphon $W\equiv q$ is always a feasible solution, so
\begin{equation}\label{eq:Phi-le-h}
  \Phi(q)\le h(q)\qquad\text{for all }q\in(0,1).
\end{equation}
Equality in \eqref{eq:Phi-le-h} is called \emph{replica symmetry}: the cheapest way to suppress copies of $H$ is to thin the edge density uniformly from $p$ to $qp$. Strict inequality is called \emph{symmetry breaking}: some inhomogeneous configuration is strictly cheaper.

The variational problem \eqref{eq:Phi} was studied systematically by Zhao \cite{ZhaoLowerTail}. He showed that for $q$ sufficiently close to $1$ the constant graphon is the unique minimizer and that, when $H$ is nonbipartite, for $q$ sufficiently close to $0$ it is not a minimizer, and he conjectured that these two regimes are always separated by a single threshold.

\begin{conjecture}[{\cite[Conjecture 2.4]{ZhaoLowerTail}}]\label{conj:zhao}
Let $H$ be a nonbipartite graph. There is $r_H^*\in(0,1)$ such that $W\equiv q$ is a minimizer of \eqref{eq:Phi} if and only if $q\ge r_H^*$, and is the unique minimizer if and only if $q>r_H^*$.
\end{conjecture}

There are two features of \cref{conj:zhao} that will be relevant below. First, it does not require fully solving the variational problem. Determining symmetry breaking optimizers is a separate problem of great interest (e.g. \cite[Section 6]{ZhaoLowerTail}). Second, the ``only if'' in the uniqueness claim asserts nonuniqueness at the threshold itself, so that multiple phases coexist there. Zhao noted this in \cite[Footnote 2]{ZhaoLowerTail}, where he observed on the basis of the shape of the phase diagram that the lower tail optimizer should change discontinuously across the boundary, in contrast to the upper tail.

In this paper we prove \cref{conj:zhao}, and we show in addition that the coexistence at the threshold takes a quantitative form: the two phases differ in edge density by an explicit positive amount. From this we deduce that the transition is of first order. A feature of our proof is that it does not require identifying the symmetry breaking minimizer. We then transfer the resulting phase diagram to the conditioned random graph in the aforementioned supercritical regime.

\subsection{Results}

Throughout the rest of the paper $H$ is a fixed nonbipartite graph, and we suppress the dependence on $H$ in the notation. Let
\begin{equation*}
  \cM_q:=\bigl\{W:t(H,W)\le q^{e(H)},\ \Ent(W)=\Phi(q)\bigr\}
\end{equation*}
denote the set of minimizers of \eqref{eq:Phi}, understood throughout modulo measure-preserving relabelings. Since every feasible graphon has entropy cost at least $\Phi(q)$, the set $\cM_q$ is the intersection of the closed set $\{t(H,\cdot)\le q^{e(H)}\}$ with the closed set $\{\Ent\le\Phi(q)\}$ inside the compact graphon space, and it is therefore compact and nonempty (see \cref{lem:lsc}). We also define
\begin{equation*}
  \kappa(q):=\frac{e(H)^2q^{2e(H)-2}}{8c_H^2\log^3(1/q)} \text{ for } q\in(0,1),
\end{equation*}
which is an increasing function of $q$ where $c_H$ is the constant of \cref{lem:approx-sidorenko}.

\begin{theorem}\label{thm:variational}
There is a unique $q_c\in(0,1)$ with the following properties.
\begin{enumerate}[label=\textup{(\roman*)},leftmargin=*]
  \item\label{it:rs} If $q>q_c$ then $\cM_q=\{W\equiv q\}$.
  \item\label{it:sb} If $q<q_c$ then $\Phi(q)<h(q)$, every $W\in\cM_q$ is nonconstant, and
    \begin{equation}\label{eq:sb-gap}
      \dens(W)\ge q+\kappa(q).
    \end{equation}
  \item\label{it:coex} At $q=q_c$ both the constant graphon $W\equiv q_c$ and at least one nonconstant graphon belong to $\cM_{q_c}$; moreover every nonconstant $W\in\cM_{q_c}$ satisfies $\dens(W)\ge q_c+\kappa(q_c)$.
\end{enumerate}
\end{theorem}

Parts \ref{it:rs}--\ref{it:coex} confirm \cref{conj:zhao} with $r^*_H=q_c$. Since $\kappa$ is increasing and \cite[Theorem 2.7]{ZhaoLowerTail} implies $q_c>0$, the density jump $\kappa(q_c)$ is always strictly positive.

Part \ref{it:coex} says that the symmetry breaking phase does not emerge continuously out of the constant phase, and the amount by which it fails to do so is measured by a jump in edge density
\begin{equation}\label{eq:Delta}
  \Delta:=\max_{W\in\cM_{q_c}}\dens(W)-q_c,
\end{equation}
where the maximum is attained because $\cM_{q_c}$ is compact and $\dens$ is continuous in the cut metric, and where $\Delta\ge\kappa(q_c)>0$. If we return to the parametrization of the rate function by the $H$-density $\eta=q^{e(H)}$ and write $\eta_c:=q_c^{e(H)}$, this jump forces a corner in $\Phi(H,\cdot)$ at $\eta_c$ whose size is controlled by $\Delta$. Here and below $\Phi'_\pm(H,\eta)$ denote the one-sided derivatives of $\Phi(H,\cdot)$ in $\eta$.

\begin{theorem}\label{thm:corner}
The function $\Phi(H,\cdot)$ coincides with $\eta\mapsto h(\eta^{1/e(H)})$ on $[\eta_c,1)$; in particular
\begin{equation}\label{eq:right-derivative}
  \Phi'_+(H,\eta_c)=\frac{\log q_c}{e(H)q_c^{e(H)-1}}.
\end{equation}
On the other side, the left derivative $\Phi'_-(H,\eta_c)$ exists and
\begin{equation}\label{eq:left-derivative}
  \Phi'_-(H,\eta_c)
  \ \ge\ \Phi'_+(H,\eta_c)+\frac{\Delta}{e(H)q_c^{e(H)}}
  \ \ge\ \Phi'_+(H,\eta_c)+\frac{e(H)q_c^{e(H)-2}}{8c_H^2\log^3(1/q_c)}.
\end{equation}
In particular $\Phi(H,\cdot)$ is not differentiable at $\eta_c$. 
\end{theorem}

Estimate \eqref{eq:left-derivative} says that the transition is first order. The existence of the left derivative is not automatic, because the feasible set $\{t(H,\cdot)\le\eta\}$ is not convex and we do not know $\Phi(H,\cdot)$ to be convex; it follows instead from a local semiconcavity property of $\Phi$ (\cref{lem:semiconcave}). 

We conclude with a probabilistic result that transfers this phase diagram to the conditioned graph. Importantly, our result keeps track of the whole set of minimizers $\cM_q$ without explicitly identifying it. For a graph $G$ on $[n]$ we write $A_G$ for its adjacency step kernel on $[0,1]^2$. The cut distance $\delta_\cut$ is defined for arbitrary integrable symmetric kernels, so that $\delta_\cut(A_G/p,W)$ makes sense even though $A_G/p$ does not take values in $[0,1]$.

\begin{theorem}\label{thm:probabilistic}
Let $p=p_n\to0$ with $p_nn^{1/m_2(H)}\to\infty$, fix $q\in(0,1)$, and let $G_{\cL}$ be distributed as $G(n,p_n)$ conditioned on $\cL_p(H,q^{e(H)})$.
\begin{enumerate}[label=\textup{(\roman*)},leftmargin=*]
  \item\label{it:sat} The constraint is asymptotically saturated:
    \[
      N_H(G_{\cL})\big/\E[N_H(p_n)]\longrightarrow q^{e(H)}
    \]
    in probability.
  \item\label{it:cut} The conditioned graph concentrates near the minimizer set:
    \begin{equation}\label{eq:cut-conditioning}
      \dist_\cut\Bigl(\frac{A_{G_{\cL}}}{p_n},\cM_q\Bigr)\longrightarrow0
    \end{equation}
    in probability.
  \item\label{it:order} The normalized edge density $e(G_{\cL})/(\binom n2 p_n)$ is a discontinuous order parameter. If $q>q_c$ it converges in probability to $q$. If $q<q_c$ then for every $\varepsilon>0$
    \begin{equation*}
      \Pr\Bigl(\frac{e(G_{\cL})}{\binom n2 p_n}\ge q+\kappa(q)-\varepsilon\Bigr)\longrightarrow1.
    \end{equation*}
    If $q=q_c$ then for every $0<\varepsilon<\kappa(q_c)/2$
    \begin{equation}\label{eq:forbidden}
      \Pr\Bigl(q_c+\varepsilon<\frac{e(G_{\cL})}{\binom n2 p_n}<q_c+\kappa(q_c)-\varepsilon\Bigr)\longrightarrow0.
    \end{equation}
  \end{enumerate}
\end{theorem}

We place no hypotheses on $\cM_q$ in \eqref{eq:cut-conditioning}: the minimizer need not be unique, and the minimizers need not be multipodal. At coexistence the conditional law is free to select the constant phase, a nonconstant phase, or a mixture of the two, and \eqref{eq:cut-conditioning} gives concentration near the union of these, while \eqref{eq:forbidden} says that no mass survives strictly between the constant phase and the nonconstant ones.

\subsection{Related work}

\subsubsection*{Dense graphs}
Chatterjee and Varadhan \cite{ChatterjeeVaradhan} proved the graphon large deviation principle for $G(n,p)$ at constant $p$, identifying the rate function and describing the conditioned structure at the same time. The resulting theory of replica symmetry in exponential and constrained random graph models was developed by Chatterjee and Diaconis \cite{ChatterjeeDiaconis}, Radin and Yin \cite{RadinYin}, and Lubetzky and Zhao \cite{LubetzkyZhao}. For edge--$k$-star constraints, Kenyon, Radin, Ren, and Sadun \cite{KenyonRadinRenSadun} proved that typical graphs are multipodal. In the edge--triangle model, Radin, Ren, and Sadun \cite{RadinRenSadun} exhibited a symmetry breaking transition, and Neeman, Radin, and Sadun \cite{NeemanRadinSadun} determined sharp singularities as the constraining densities vary between curves of extremal values. DiCarlo and Sadun \cite{DiCarloSadun} illustrated the complexity of that phase diagram numerically, finding tripodal rather than bipodal structure in part of the undersaturated region, together with a discontinuous transition at its boundary. In this paper we study the sparse analog of this phenomenon, in a regime where the optimizers are not known in general.

\subsubsection*{Sparse subgraph counts}
When $p=o(1)$ the dense theory does not apply directly, and a different set of techniques is required. Chatterjee and Dembo \cite{ChatterjeeDembo} extended the variational asymptotics to the range $p\ge n^{-\alpha_H}$, and this inspired a line of work developing the technique and the associated regularity theory considerably further, by Eldan \cite{Eldan}, Cook and Dembo \cite{CookDembo}, Augeri \cite{Augeri}, and Cook, Dembo, and Pham \cite{CookDemboPham}. For lower tails, Kozma and Samotij \cite{KozmaSamotij} reached the full mean field range $p\gg n^{-1/m_2(H)}$ by a relative entropy increment argument, which can be interpreted as a probabilistic counterpart of the hypergraph container theorems of Balogh, Morris, and Samotij \cite{BaloghMorrisSamotij} and Saxton and Thomason \cite{SaxtonThomason}. Below the mean field range the lower tail probability exhibits Poisson-type asymptotics under fairly general hypotheses; see Janson and Warnke \cite{JansonWarnke}.

\subsubsection*{The variational problem and conditional structure}
Far less is known about the variational problem \eqref{eq:Phi} itself, and for a long time the only results were those proven by Zhao \cite{ZhaoLowerTail}, who identified a neighborhood of $q=1$ on which the constant graphon is the unique minimizer and a neighborhood of $q=0$ on which it is not a minimizer. Very recent work of Kwan and Pham \cite{KwanPham} implies a two-valued multipodal optimizer for sufficiently small $q$ and strictly 2-balanced edge-critical graphs $H$. We established sparse conditional structure in \cite{ChinStructure} under the hypothesis that the minimizer is known to be the constant graphon and to be unique, which confines $q$ to a proper subinterval of the replica symmetric phase, by combining the entropy increment scheme of \cite{KozmaSamotij} with a stability version of Zhao's argument. \Cref{thm:probabilistic}\ref{it:cut} extends this result to all of $(q_c,1)$, the entire symmetry breaking interval $(0,q_c)$, and the coexistence point $q_c$. To the best of our knowledge, this is the first result covering the symmetry breaking regime at fixed positive lower tails deficits in the supercritical regime. Our proof also works differently: we replace stability around a single known graphon by compact stability around the whole set of minimizers, and we pass from the conditional marginals to the sampled graph by a relative entropy projection rather than by subgraph counts.

\subsubsection*{The critical window}
The critical regime $p=\Theta(n^{-1/m_2(H)})$ has seen exciting developments in recent years. At this scaling, the mean field variational problem no longer captures the lower tail rate function, and the answer depends on the constant $c$ in $p=cn^{-1/m_2(H)}$ as well as on the deficit $\eta$. First for triangles, Jenssen, Perkins, Potukuchi, and Simkin \cite{JenssenPerkinsPotukuchiSimkin} imported methods from statistical physics to study a triangle-penalized Gibbs measure (also known as an exponential random graph) and obtained critical window asymptotics for all $c$ when $\eta\ge0.4993$ and for sufficiently small $c$ otherwise, finding a transition in $c$ for small deficits but none for $\eta\ge0.4993$. Their subsequent work \cite{JenssenPerkinsPotukuchiSimkinBP} obtained Bethe-type asymptotics for a broader range of parameters and more general graphs $H$, subject to a hard-core model upper bound on $p$. Very recently, their newest work \cite{JenssenPerkinsPotukuchiSimkinRS} expresses the lower tail rate function for strictly 2-balanced graphs $H$ and $\eta = 0$ as a variational problem optimizing a cost that interpolates smoothly between the mean-field entropy cost of the supercritical regime and the Poisson cost of the subcritical regime. The same variational problem was derived concurrently and more generally for lower tail events by Kwan and Pham \cite{KwanPham} via more combinatorial methods. 

\subsection{Proof overview}

The proof splits into two largely independent halves, which we discuss separately.

\subsubsection*{Analysis of the variational problem}
The main ingredient of this half is a new estimate at the constant graphon, which we now describe. The idea is to show that the constant graphon is always a (quantitatively) local minimizer of the variational problem. Strong convexity of $h$ gives, for every graphon $W$ with $\sigma:=\|W-q\|_2$ and $m:=\dens(W)-q$,
\begin{equation}\label{eq:overview-coercive-1}
  \Ent(W)-h(q)\ \ge\ -\log(1/q)\,m+\tfrac12\sigma^2 .
\end{equation}
On the other hand, expanding $t(H,\cdot)$ around the constant graphon $\mu:=\dens(W)$ gives the approximate Sidorenko inequality
\[
  \mu^{e(H)}-t(H,W)\ \le\ c_H\|W-\mu\|_2^3
\]
of \cref{lem:approx-sidorenko}. For triangles the exact expansion, with $V:=W-\mu$, reads
\[
  t(K_3,\mu+V)-\mu^3=3\mu\,t(P_2,V)+t(K_3,V);
\]
it has no first order term because $\int V=0$, a nonnegative second order term by Cauchy--Schwarz, and a cubic term which is controlled by the bound $|t(K_3,V)|\le\|V\|_2^3$. For general $H$ the star terms group to give a nonnegative contribution, and all remaining terms are again cubic. Since $\|W-\mu\|_2\le\sigma$ and $\mu^{e(H)}-q^{e(H)}\ge e(H)q^{e(H)-1}m$ when $m>0$, feasibility therefore forces $m\le c_H\sigma^3/(e(H)q^{e(H)-1})$, and substituting this into \eqref{eq:overview-coercive-1} yields
\begin{equation}\label{eq:overview-coercive-2}
  \Ent(W)-h(q)\ \ge\ \tfrac12\sigma^2-\frac{c_H\log(1/q)}{e(H)q^{e(H)-1}}\sigma^3
  \qquad\text{whenever } t(H,W)\le q^{e(H)} .
\end{equation}
The estimate \eqref{eq:overview-coercive-2} shows that any competitor of the constant graphon must be macroscopically far from it, and \eqref{eq:overview-coercive-1} then converts that distance into the edge density gap \eqref{eq:sb-gap}. 

The analysis of the phase transition centers on the scaling identity
\begin{equation*}
  \Ent(sW)-h(sq)=s\bigl(\Ent(W)-h(q)\bigr)+s\log s\bigl(\dens(W)-q\bigr),
  \qquad t(H,sW)=s^{e(H)}t(H,W)
\end{equation*}
which was already observed by Zhao \cite{ZhaoLowerTail}. A competitor $W$ to the constant graphon $q$ with $\dens(W)>q$ therefore dilates to a strictly better competitor at every smaller parameter, so the symmetry breaking set must be an interval $(0,q_c)$. The local minimizer estimate supplies the uniform gap $\dens(W)\ge q+\kappa(q)$ that is needed to pass to a subsequential limit along $q\uparrow q_c$ while remaining far from the constant graphon to prove coexistence. Dilating the resulting critical minimizer below $q_c$ produces an explicit competing branch $C$ whose left derivative at $q_c$ exceeds that of the replica symmetric branch by exactly $\Delta/q_c$, and comparing $C$ with $\Phi$ then yields the separation of one-sided derivatives.

A natural alternative route to the rate function, used by \cite{JenssenPerkinsPotukuchiSimkin,JenssenPerkinsPotukuchiSimkinBP}, is to introduce a triangle-penalized Gibbs measure and to estimate its edge density. At a first order transition this is delicate: a canonical ensemble tuned to the critical penalty may mix macroscopically distinct phases, so concentration of any subgraph or edge density is not available a priori. Our argument works directly with the rate function instead, bounding the derivatives by dilating potential optimizers. 

\subsubsection*{Transfer to the conditioned graph}
To analyze the conditioned graph we use the entropy decomposition technique of Kozma and Samotij \cite{KozmaSamotij}. This expresses the complicated conditional measure as a small mixture of nicer measures, each of which can be related to the optimizers of the variational problem directly. 

Write $\nu_{n,q}$ for the law of $G_{\cL}$. For the structural statements our proof proceeds as follows: given the conditional law $\nu_{n,q}$, the entropy increment lemma (\cref{thm:KS}\ref{it:KS-increment}) produces a small set $S$ of pinned edges such that, after conditioning on $(Y_e)_{e\in S}$, the one-edge marginals $\mathbf a$ approximately satisfy the lower tail condition. A compactness argument (\cref{lem:graphon-gap}) shows that such vectors whose rescaled step graphons stay at cut distance at least $\varepsilon$ from $\cM_q$ cost strictly more entropy than $-\log\Pr(\cL_p(H,q^{e(H)}))$ allows, and from this we deduce that almost every conditional component has marginals close to $\cM_q$.

Passing from the marginals to the sampled graph requires a second ingredient. Let $Q_s$ be the product measure with the conditional marginals and let $\nu_s$ be the true conditional law of the unpinned edges. The relative entropy $\KL(\nu_s\|Q_s)$ is the conditional total correlation, and the chain rule for relative entropy shows that its average is $o(n^2p)$, again by comparison with the lower tail rate. Under $Q_s$, Bernstein's inequality and a union bound over cuts give cut concentration at scale $\varepsilon p$ with failure probability $e^{-c\varepsilon^2n^2p}$. A relative entropy inequality then transfers this bound from $Q_s$ to $\nu_s$ at a cost of a factor $\KL(\nu_s\|Q_s)/(c\varepsilon^2n^2p)=o(1)$. Finally, the pinned edges carry expected mass $O(p|S|)$ by conditional domination, so that they are negligible in the $p$-rescaled cut norm since $S$ is small. This proves \eqref{eq:cut-conditioning}. Part \ref{it:sat} follows from strict monotonicity of $\Phi$, while Part \ref{it:order} then follows using the variational input of \cref{thm:variational}.

\subsection{Organization}

\Cref{sec:prelim} collects the graphon, variational problem, and entropy increment preliminaries. \Cref{sec:coercivity} proves the new local minimizer estimate. In \cref{sec:transition} we establish~\cref{thm:variational} and~\cref{thm:corner} -- the single threshold, coexistence, and the corner. \Cref{sec:conditioning} proves the conditioning theorem, and we conclude in \cref{sec:consequences} deducing saturation and the order parameter statements to complete~\cref{thm:probabilistic}.

\subsection{Notation}

Throughout, $\|\cdot\|_2$ and $\|\cdot\|_1$ denote $L^2$ and $L^1$ norms on $[0,1]^2$ with Lebesgue measure, and $\|\cdot\|_\cut$ the cut norm \eqref{eq:cut-norm}. We write $P_2$ for the path with two edges, so that
\[
  t(P_2,U)=\int_0^1\Bigl(\int_0^1U(x,y)\dd y\Bigr)^2\dd x .
\]
We write $v(H)$ and $e(H)$ for the numbers of vertices and edges of $H$, and
\[
  m_2(H):=\max\Bigl\{\frac{e(F)-1}{v(F)-2}:F\subseteq H,\ e(F)\ge2\Bigr\}
\]
for its $2$-density; since $H$ contains an odd cycle, $m_2(H)>1$. We write $\binom{[n]}2$ for the edge set of $K_n$. A constant $q\in[0,1]$ is identified with the constant graphon of that value. Asymptotic notation refers to $n\to\infty$ with $H$ and $q$ fixed, and $p=p_n$ always satisfies $p_n\to0$. Finally, $\log$ is the natural logarithm and $0\log0:=0$.

\subsection*{Acknowledgments}
The author thanks Sam van der Poel for sharing ~\cref{lem:approx-sidorenko} and~\cref{lem:semiconcave} to upgrade from triangles to non-bipartite $H$, and Ashwin Sah for helpful initial discussions.

\subsection*{AI Usage}
The main new ingredient, presented in Section~\ref{sec:coercivity}, and Theorems~\ref{thm:variational}\ref{it:rs}, \ref{thm:variational}\ref{it:sb} and~\ref{thm:probabilistic}\ref{it:order} were proven by the author for the case $H = K_3$ in 2025. ChatGPT, given a draft of these as input, was used for further discussion which led to the remaining results. Claude was used to assist in revising the manuscript. The author takes full responsibility for the content of the paper.

\section{Preliminaries}\label{sec:prelim}

\subsection{Graphons, cut distance, and the entropy cost}
Let $\cW$ be the set of symmetric measurable functions $W\colon[0,1]^2\to[0,1]$. For an integrable symmetric kernel $U$ put
\begin{equation}\label{eq:cut-norm}
  \|U\|_\cut:=\sup_{S,T\subseteq[0,1]}\Bigl|\int_{S\times T}U\Bigr|,
\end{equation}
the supremum over measurable sets, and for measure-preserving bijections $\phi$ of $[0,1]$ write \\ $W^\phi(x,y):=W(\phi(x),\phi(y))$ and
\begin{equation}\label{eq:cut-distance}
  \delta_\cut(U,W):=\inf_\phi\|U-W^\phi\|_\cut .
\end{equation}
We identify kernels at cut distance zero; the resulting space $\widetilde\cW$ of graphons is compact, and homomorphism densities are continuous on it. We use these facts in the form of \cite[Chapters 8--9]{LovaszBook}. Both $\Ent$ and $\dens$ are relabeling invariant and hence well defined on $\widetilde\cW$, and $\dens$ is $1$-Lipschitz for $\delta_\cut$ because $|\dens(U)-\dens(V)|\le\|U-V\|_\cut$. We will repeatedly use the elementary chain
\begin{equation}\label{eq:norm-chain}
  \|U\|_\cut\le\|U\|_1\le\|U\|_2 .
\end{equation}
For a constant target the infimum in \eqref{eq:cut-distance} is superfluous: $\delta_\cut(W,q)=\|W-q\|_\cut$.

We record without proof standard but useful properties of graphon space and their implications for the variational problem. First, the entropy cost is not continuous in the cut metric, but it is lower semicontinuous, and this is all that we will need.

\begin{lemma}\label{lem:lsc}
If $\delta_\cut(W_k,W)\to0$ then $\Ent(W)\le\liminf_k\Ent(W_k)$. Consequently the minimum in \eqref{eq:Phi} is attained for every $\eta\in[0,1]$.
\end{lemma}

This implies the following nice properties of the full variational problem. 
\begin{lemma}\label{lem:Phi-basic}
The function $\Phi$ is continuous and strictly decreasing on $(0,1)$, and every $W\in\cM_q$ with $q\in(0,1)$ satisfies $t(H,W)=q^{e(H)}$.
\end{lemma}

Finally, the following is an elementary observation which will be used repeatedly.

\begin{lemma}\label{lem:constants}
Let $q\in(0,1)$. A constant graphon $W\equiv c$ is feasible for $\Phi(q)$ if and only if $c\le q$, and in that case $\Ent(W)=h(c)\ge h(q)$, with equality if and only if $c=q$. Consequently, if $\Phi(q)<h(q)$ then no minimizer at $q$ is constant, while if $\Phi(q)=h(q)$ then the only constant minimizer is $W\equiv q$.
\end{lemma}

\subsection{The finite mean field problem and its sparse limit}
For completeness, we provide a proof that the sparse limit of the variational problem is rigorously tied to the finite dimensional mean-field approximation used by Kozma and Samotij \cite{KozmaSamotij}.

For $p\in(0,1)$ let $\ip(x)=x\log(x/p)+(1-x)\log\bigl((1-x)/(1-p)\bigr)$ and, for $\eta\ge0$,
\begin{equation}\label{eq:Psi}
  \Phi_p(H,\eta):=\min\Bigl\{\sum_{e\in\binom{[n]}2}\ip(x_e)\ :\ \E[N_H(\mathbf x)]\le\eta\,\E[N_H(p)]\Bigr\},
\end{equation}
the minimum over $\mathbf x\in[0,1]^{\binom{[n]}2}$, as in \cite{ChinStructure}. Here $N_H(\mathbf x)$ denotes the number of copies of $H$ in the random graph containing each edge $e$ independently with probability $x_e$, so that
\[
  \E[N_H(\mathbf x)]=\sum_{A\in\cH(H)}\prod_{e\in A}x_e,
\]
where $\cH(H)$ is the $e(H)$-uniform hypergraph with vertex set $\binom{[n]}2$ whose edges are the copies of $H$ in $K_n$, each identified with its edge set. Recall also $\Phi(H,\eta)$ from \eqref{eq:Phi}, so that $\Phi(H,\eta)=\Phi(\eta^{1/e(H)})$ for $\eta\in[0,1]$ and $\Phi(H,\eta)=0$ for $\eta\ge1$. Given $\mathbf x\in[0,1]^{\binom{[n]}2}$ let $W_{\mathbf x}$ be the step kernel with $n$ equal parts $I_i=\bigl((i-1)/n,\,i/n\bigr]$ taking the value $x_{ij}$ on $I_i\times I_j$ for $i\ne j$ and the value $0$ on the diagonal blocks. Two relations will be used repeatedly:
\begin{equation}\label{eq:step-identities}
  \Bigl|t(H,W_{\mathbf x})-\frac{\E[N_H(\mathbf x)]}{N_H}\Bigr|\le\frac{\binom{v(H)}2}n,
  \qquad
  \Ent(W_{\mathbf x})=\Bigl(1-\frac1n\Bigr)\binom n2^{-1}\!\!\sum_{e\in \binom{[n]}2}\!\!h(x_e)\ +\ \frac1n .
\end{equation}
For the first, write $t(H,W_{\mathbf x})$ as the average over all maps $V(H)\to[n]$ of the corresponding products of entries of $\mathbf x$, with the convention that the diagonal entries vanish. Each copy of $H$ arises from exactly $|\mathrm{Aut}(H)|$ injective maps, so the injective maps contribute $\prod_{i<v(H)}(1-i/n)$ times $\E[N_H(\mathbf x)]/N_H\in[0,1]$, while the noninjective maps form a proportion at most $\binom{v(H)}2/n$ of all maps and contribute a number in $[0,\binom{v(H)}2/n]$. The second is immediate from $h(0)=1$.

\begin{proposition}\label{prop:finite-limit}
Let $p_n\to0$. For every $\eta\in(0,1)$,
\begin{equation}\label{eq:finite-limit}
  \frac{\Phi_{p_n}(H,\eta)}{\binom n2p_n}\longrightarrow \Phi(H,\eta).
\end{equation}
Moreover, the lower bound is stable under the following cut-closed restriction: if $\eta_n\to\eta\in(0,1)$, if $\varepsilon>0$, and if $\cK\subseteq\widetilde\cW$ is closed, then
\begin{multline}\label{eq:restricted-liminf}
  \liminf_{n\to\infty}\frac1{\binom n2p_n}
  \inf\Bigl\{\sum_e\ip(x_e)\ :\ \mathbf x\in[0,p_n]^{\binom{[n]}2},\\
  \E[N_H(\mathbf x)]\le\eta_n\E[N_H(p_n)],\
  \dist_\cut\bigl(W_{\mathbf x/p_n},\cK\bigr)\ge\varepsilon\Bigr\}
\end{multline}
is at least $\inf\{\Ent(W):t(H,W)\le\eta,\ \dist_\cut(W,\cK)\ge\varepsilon\}$. Here $\dist_\cut(\cdot,\cK):=\inf_{V\in\cK}\delta_\cut(\cdot,V)$ and $\inf\varnothing:=+\infty$.
\end{proposition}

\begin{proof}
We first record two elementary facts. Since $\ip$ is decreasing on $[0,p]$ and increasing on $[p,1]$, and since decreasing a coordinate decreases the polynomial in \eqref{eq:Psi}, a minimizer of \eqref{eq:Psi} may be assumed to satisfy $x_e\le p$ for every $e$. Second, writing $x=pw$ with $w\in[0,1]$,
\begin{equation}\label{eq:ip-expansion}
  \ip(pw)=pw\log w+(1-pw)\bigl(\log(1-pw)-\log(1-p)\bigr)=p\,h(w)+O(p^2)
\end{equation}
uniformly in $w\in[0,1]$, because $\log(1-pw)-\log(1-p)=p(1-w)+O(p^2)$ uniformly.

We first prove the lower bound, including \eqref{eq:restricted-liminf}. Pass to a subsequence realizing the lower limit. If that lower limit is $+\infty$, the claim is immediate; otherwise the restricted feasible set is nonempty for all sufficiently large indices of this subsequence. Let $\mathbf x^{(n)}\in[0,p_n]^{\binom{[n]}2}$ satisfy the constraints in \eqref{eq:restricted-liminf} and be within $o(\binom n2p_n)$ of the infimum, and set $\mathbf w^{(n)}=\mathbf x^{(n)}/p_n$ and $U_n:=W_{\mathbf w^{(n)}}$. By \eqref{eq:ip-expansion} and \eqref{eq:step-identities},
\begin{equation}\label{eq:cost-vs-entropy}
  \frac1{\binom n2p_n}\sum_e\ip(x^{(n)}_e)=\binom n2^{-1}\sum_eh(w^{(n)}_e)+O(p_n)=\Ent(U_n)+O(n^{-1})+O(p_n).
\end{equation}
By \eqref{eq:step-identities} the constraint gives $t(H,U_n)\le\eta_n+\binom{v(H)}2/n$. By compactness, pass to a further subsequence along which $U_n\to U$ in cut distance. Then $t(H,U)\le\eta$ by continuity, $\dist_\cut(U,\cK)\ge\varepsilon$ because $\dist_\cut(\cdot,\cK)$ is $1$-Lipschitz for $\delta_\cut$, and $\Ent(U)\le\liminf_n\Ent(U_n)$ by \cref{lem:lsc}. Combining with \eqref{eq:cost-vs-entropy} gives the claim. Taking $\cK=\varnothing$, so that the last constraint is vacuous, gives the lower bound in \eqref{eq:finite-limit}.

We now prove the upper bound in \eqref{eq:finite-limit}. Fix $\varepsilon>0$. Let $W^*$ attain $\Phi(H,\eta)$. For $s\in(0,1)$ the graphon $sW^*$ satisfies $t(H,sW^*)=s^{e(H)}t(H,W^*)<\eta$, and $\Ent(sW^*)\to\Ent(W^*)$ as $s\uparrow1$ by bounded convergence; fix $s$ so close to $1$ that $W:=sW^*$ has $\Ent(W)\le \Phi(H,\eta)+\varepsilon$ and $t(H,W)<\eta$. Let $\cP_n$ be the partition of $[0,1]$ into the $n$ intervals $I_i$ and let $x_e:=p_n\cdot(\text{average of }W\text{ over the block of }e)$, so that $W_{\mathbf x/p_n}$ agrees with $W_{\cP_n}$ off the diagonal blocks and vanishes on them. Then $W_{\mathbf x/p_n}\le W_{\cP_n}$ pointwise, so
\[
  t(H,W_{\mathbf x/p_n})\le t(H,W_{\cP_n})\xrightarrow[n\to\infty]{}t(H,W)<\eta ,
\]
using standard equal-grid approximation, which gives $W_{\cP_n}\to W$ in $L^1$, and the counting lemma. By \eqref{eq:step-identities} the vector $\mathbf x$ is therefore feasible for $\Phi_{p_n}(H,\eta)$ once $n$ is large. Its cost is, by \eqref{eq:ip-expansion}, \eqref{eq:step-identities} and Jensen's inequality,
\begin{align*}
  \sum_e\ip(x_e)&=\binom n2p_n\Bigl(\binom n2^{-1}\sum_eh\bigl(x_e/p_n\bigr)+O(p_n)\Bigr)\\
  &\le\binom n2p_n\Bigl(\tfrac{n}{n-1}\Ent(W_{\cP_n})+O(p_n)\Bigr)
  \ \le\ \binom n2p_n\bigl(\Ent(W)+o(1)\bigr).
\end{align*}
Hence $\limsup_n\Phi_{p_n}(H,\eta)/(\binom n2p_n)\le \Phi(H,\eta)+\varepsilon$, and $\varepsilon>0$ was arbitrary.
\end{proof}

\subsection{The relative entropy input}

Finally, in this subsection we collect a series of results from the work of Kozma and Samotij \cite{KozmaSamotij}, specialized to copies of $H$. The hypergraph $\cH(H)$ satisfies the degree condition of \cite[\S7]{KozmaSamotij} as soon as $p\ge Ln^{-1/m_2(H)}$ for a large constant $L$. Let $Y=(Y_e)_{e\in \binom{[n]}2}$ be the edge-indicator vector of a random graph and, for an edge set $A$, write $Y_A:=\prod_{e\in A}Y_e$. For a set $S\subseteq \binom{[n]}2$, $b\in B\subseteq A\in\cH(H)$ put
\begin{equation*}
  D_S(B,b):=\E\bigl[Y_B\mid Y_S\bigr]-\E\bigl[Y_{B\setminus\{b\}}\mid Y_S\bigr]\E\bigl[Y_b\mid Y_S\bigr],
\end{equation*}
where $Y_S$ abbreviates $(Y_e)_{e\in S}$, and set
\begin{equation*}
  \cE_H(S):=\E\Biggl[\ \sum_{A\cap S=\varnothing}\ \sum_{\substack{B\subseteq A\\|B|\ge2}}\ \sum_{b\in B}\frac{D_S(B,b)^2}{\binom{e(H)}{|B|}|B|\,p^{2|B|}}\Biggr].
\end{equation*}

\begin{theorem}[Kozma--Samotij \cite{KozmaSamotij}]\label{thm:KS}
Fix $p_0\in(0,1)$. In parts \ref{it:KS-harris}--\ref{it:KS-increment}, let $Y$ be distributed as $G(n,p)$ conditioned on a decreasing event.
\begin{enumerate}[label=\textup{(\alph*)},leftmargin=*]
  \item\label{it:KS-tail} For every $\varepsilon\in(0,1)$ there is $L=L(\varepsilon,p_0)$ such that if $Ln^{-1/m_2(H)}\le p\le p_0$ and $\eta\in[0,1]$ then
  \begin{equation}\label{eq:KS-tail}
    (1-\varepsilon)\Phi_p(H,\eta+\varepsilon)\ \le\ -\log\Pr\bigl(\cL_p(H,\eta)\bigr)\ \le\ (1+\varepsilon)\Phi_p\bigl(H,(1-\varepsilon)\eta\bigr).
  \end{equation}
  \item\label{it:KS-harris} For disjoint edge sets $A,S$, one has $\E[Y_A\mid Y_S]\le p^{|A|}$ almost surely.
  \item\label{it:KS-CS} With the notation above,
  \begin{equation}\label{eq:KS-CS}
    \E\sum_{A\cap S=\varnothing}\Bigl|\E[Y_A\mid Y_S]-\prod_{e\in A}\E[Y_e\mid Y_S]\Bigr|\ \le\ p^{e(H)}\bigl((e(H)-1)N_H\bigr)^{1/2}\cE_H(S)^{1/2}.
  \end{equation}
  \item\label{it:KS-increment} For all $\alpha,\beta>0$ there is $L=L(\alpha,\beta)$ such that if $p\ge Ln^{-1/m_2(H)}$ and $n$ is large, then there exists $S\subseteq \binom{[n]}2$ with $|S|\le\alpha\binom n2$ and $\cE_H(S)\le\beta N_H$.
\end{enumerate}
\end{theorem}

Part \ref{it:KS-tail} is \cite[Theorem 2]{KozmaSamotij}; part \ref{it:KS-harris} is \cite[Claim 16]{KozmaSamotij}, an application of Harris's inequality; part \ref{it:KS-CS} is the Cauchy--Schwarz estimate of \cite[p.~686]{KozmaSamotij}; and part \ref{it:KS-increment} is \cite[Lemma 18]{KozmaSamotij} combined with \cite[\S7]{KozmaSamotij}. All four are recorded in this generality in \cite[\S3]{ChinStructure}. These tools relate the lower tail rate function to the variational problem in the supercritical regime as follows. 

\begin{corollary}\label{cor:ldp}
If $p_n\to0$ and $p_nn^{1/m_2(H)}\to\infty$, then for every $q\in(0,1)$,
\begin{equation*}
  -\frac1{\binom n2p_n}\log\Pr\bigl(\cL_p(H,q^{e(H)})\bigr)\longrightarrow\Phi(q).
\end{equation*}
\end{corollary}

\begin{proof}
Fix $\varepsilon\in(0,1-q^{e(H)})$. Since $p_nn^{1/m_2(H)}\to\infty$, \cref{thm:KS}\ref{it:KS-tail} applies with parameter $\varepsilon$ for all large $n$. Dividing \eqref{eq:KS-tail} by $\binom n2p_n$ and applying \cref{prop:finite-limit} gives
\begin{align*}
  (1-\varepsilon)\Phi(H,q^{e(H)}+\varepsilon)
  &\le\liminf_n\frac{-\log\Pr(\cL_p(H,q^{e(H)}))}{\binom n2p_n}\\
  &\le\limsup_n\frac{-\log\Pr(\cL_p(H,q^{e(H)}))}{\binom n2p_n}\le(1+\varepsilon)\Phi\bigl(H,(1-\varepsilon)q^{e(H)}\bigr).
\end{align*}
Now let $\varepsilon\downarrow0$ and use continuity of $\Phi(H,\cdot)$, which is \cref{lem:Phi-basic}.
\end{proof}

\section{Local minimizer at the constant graphon}\label{sec:coercivity}

In this section we prove the deterministic core of the paper, which is a quantitative statement to the effect that the constant graphon is a strict local minimum of \eqref{eq:Phi} in a scale-invariant sense.

The only property of $H$ that enters this section is the following approximate form of Sidorenko's inequality, which bounds the deficit of $t(H,W)$ below $\dens(W)^{e(H)}$ by a cubic $L^2$ error. The expansion of subgraph densities around a constant was used by Sidorenko \cite{Sidorenko} and later by Lov\'asz \cite{LovaszSigned}, and a cubic $L^2$ error of this kind was observed in a similar setting by Radin and Sadun \cite[p.~860]{RadinSadun}. For general $H$ our proof follows the star expansion of Kenyon, Radin, Ren, and Sadun \cite{KenyonRadinRenSadunBipodal}.

\begin{lemma}\label{lem:approx-sidorenko}
For every $W\in\cW$,
\begin{equation}\label{eq:approx-sidorenko}
  \dens(W)^{e(H)}-t(H,W)\ \le\ c_H\,\|W-\dens(W)\|_2^3,
\end{equation}
where $c_H:=2^{e(H)}$ if $H\neq K_3$ and $c_{K_3}:=1$.
\end{lemma}

\begin{proof}
Put $\mu:=\dens(W)$ and $V:=W-\mu$, so that $\int V=0$ and $|V|\le1$. For $F\subseteq E(H)$ write
\[
  t(F,V):=\int_{[0,1]^{V(H)}}\prod_{ij\in F}V(x_i,x_j)\prod_{i\in V(H)}\dd x_i .
\]
Expanding the product in the definition of $t(H,\mu+V)$, as in \cite[Lemma 5.1]{KenyonRadinRenSadunBipodal}, gives
\begin{equation}\label{eq:star-expansion}
  t(H,W)=\sum_{F\subseteq E(H)}\mu^{e(H)-|F|}\,t(F,V).
\end{equation}
The term $F=\varnothing$ is $\mu^{e(H)}$, and every $F$ containing an edge which shares no vertex with the other edges of $F$ contributes zero, since the integral then factors and $\int V=0$.

Suppose first that $H=K_3$. The remaining terms of \eqref{eq:star-expansion} come from the three pairs of edges, which each share a vertex and therefore each contribute $t(P_2,V)=\int_0^1\bigl(\int_0^1V(x,y)\dd y\bigr)^2\dd x\ge0$, and from $F=E(K_3)$, so that
\begin{equation}\label{eq:triangle-expansion}
  t(K_3,W)-\mu^3=3\mu\,t(P_2,V)+t(K_3,V).
\end{equation}
For the cubic term we use the bound
\begin{equation*}
  |t(K_3,V)|\le\|V\|_2^3,
\end{equation*}
which follows from \cite[Proposition 3.1]{ZhaoLowerTail} applied to $\pm V$. Since $t(P_2,V)\ge0$, this gives \eqref{eq:approx-sidorenko} with $c_{K_3}=1$.

For general $H$ we follow the proof of \cite[Theorem 4.1]{LovaszSigned} and group the stars. Let $f(x):=\int_0^1V(x,y)\dd y$, so that $\int f=0$ and $\mu+f(x)=\int_0^1W(x,y)\dd y\ge0$. For a vertex $v$ of $H$ of degree $d(v)\ge2$, the sets $F$ consisting of at least two edges at $v$ contribute
\[
  \sum_{k=2}^{d(v)}\binom{d(v)}k\mu^{e(H)-k}\int_0^1f^k
  =\mu^{e(H)-d(v)}\Bigl(\int_0^1(\mu+f)^{d(v)}-\mu^{d(v)}\Bigr)\ \ge\ 0
\]
to \eqref{eq:star-expansion}, using $\int f=0$ and Jensen's inequality. Every remaining $F$ with a nonzero contribution contains a subgraph $J$ isomorphic to $K_3$, to the path $P_3$ with three edges, or to the disjoint union $P_2\sqcup K_2$: a connected $F$ which is not a star contains $K_3$ or $P_3$, while a disconnected $F$ has at least two components, each with at least two edges. Since $|V|\le1$ and every vertex of $J$ has degree at most two, the generalized Cauchy--Schwarz inequality \cite[Lemma 3.3]{LovaszSigned} gives
\[
  |t(F,V)|\ \le\ t(J,|V|)\ \le\ \|V\|_2^3 .
\]
There are at most $2^{e(H)}$ remaining sets $F$, each with coefficient $\mu^{e(H)-|F|}\le1$, so $t(H,W)\ge\mu^{e(H)}-2^{e(H)}\|V\|_2^3$.
\end{proof}

We use this approximate Sidorenko inequality to deduce quantitative bounds on the entropy cost of a graphon $W$ in terms of its variance. 
\begin{proposition}\label{prop:coercivity}
Fix $q\in(0,1)$ and let $W\in\cW$. Write $m:=\dens(W)-q$ and $\sigma:=\|W-q\|_2$. Then
\begin{equation}\label{eq:coercive-1}
  \Ent(W)-h(q)\ \ge\ -\log(1/q)\,m+\tfrac12\sigma^2 .
\end{equation}
If moreover $t(H,W)\le q^{e(H)}$, then
\begin{equation}\label{eq:m-bound}
  m\le\frac{c_H\sigma^3}{e(H)q^{e(H)-1}}
  \qquad\text{and hence}\qquad
  \Ent(W)-h(q)\ \ge\ \tfrac12\sigma^2-\frac{c_H\log(1/q)}{e(H)q^{e(H)-1}}\,\sigma^3 .
\end{equation}
\end{proposition}

\begin{proof}
We first claim that
\begin{equation}\label{eq:h-strong-convexity}
  h(x)\ \ge\ h(q)+\log q\,(x-q)+\tfrac12(x-q)^2,\qquad x\in[0,1].
\end{equation}
To see this, let $g(x)$ denote the difference of the two sides. Then $g(q)=g'(q)=0$ and $g''(x)=x^{-1}-1\ge0$ on $(0,1]$, so $g\ge0$ on $(0,1]$; the value at $x=0$ follows by continuity. Integrating \eqref{eq:h-strong-convexity} over $[0,1]^2$ with $x=W(u,v)$ gives \eqref{eq:coercive-1}, since $\log q=-\log(1/q)$.

Now suppose $t(H,W)\le q^{e(H)}$. If $m\le0$ there is nothing to prove, so let $m>0$ and put $\mu:=\dens(W)=q+m$. Since $\int(W-\mu)=0$, we have $\sigma^2=\|W-\mu\|_2^2+m^2\ge\|W-\mu\|_2^2$, and \cref{lem:approx-sidorenko} gives
\[
  \mu^{e(H)}-q^{e(H)}\ \le\ \mu^{e(H)}-t(H,W)\ \le\ c_H\sigma^3 .
\]
By convexity of $x\mapsto x^{e(H)}$ the left-hand side is at least $e(H)q^{e(H)-1}m$, which gives the first part of \eqref{eq:m-bound}. Substituting into \eqref{eq:coercive-1} gives the second.
\end{proof}

When the graphon $W$ both satisfies the lower tail event and has low entropy cost, the above implies that it must have large variance, in particular it is far from constant. The graphon must then compensate by increasing in density to drive down its entropy cost.
\begin{corollary}\label{cor:gap}
Fix $q\in(0,1)$ and suppose that $W\in\cW$ satisfies
\begin{equation}\label{eq:competitor}
  t(H,W)\le q^{e(H)},\qquad \Ent(W)\le h(q),\qquad W\not\equiv q .
\end{equation}
Then
\begin{equation}\label{eq:three-gaps}
  \|W-q\|_2\ \ge\ \frac{e(H)q^{e(H)-1}}{2c_H\log(1/q)},
  \qquad
  \dens(W)-q\ \ge\ \kappa(q)=\frac{e(H)^2q^{2e(H)-2}}{8c_H^2\log^3(1/q)},
  \qquad
  \|W-q\|_\cut\ \ge\ \kappa(q).
\end{equation}
In particular, if $\kappa(q)>1-q$ then no such $W$ exists, so $\cM_q=\{W\equiv q\}$.
\end{corollary}

\begin{proof}
Write $\sigma=\|W-q\|_2>0$ and $m=\dens(W)-q$. The second inequality in \eqref{eq:m-bound} and $\Ent(W)\le h(q)$ give $0\ge\tfrac12\sigma^2-\bigl(c_H\log(1/q)/(e(H)q^{e(H)-1})\bigr)\sigma^3$, and dividing by $\sigma^2>0$ yields the first bound in \eqref{eq:three-gaps}. Next, \eqref{eq:coercive-1} and $\Ent(W)\le h(q)$ give $m\ge\sigma^2/(2\log(1/q))$, and substituting the first bound gives the second. The third follows from $\|W-q\|_\cut\ge|\int(W-q)|=m$, obtained by taking $S=T=[0,1]$ in \eqref{eq:cut-norm}. Finally, $\dens(W)\le1$ always, so if $\kappa(q)>1-q$ the second bound is impossible and no $W$ satisfies \eqref{eq:competitor}; then $\Phi(q)=h(q)$, and $\cM_q=\{W\equiv q\}$ by \cref{lem:constants}.
\end{proof}

A second consequence is an explicit entropy barrier. It does not follow immediately from \eqref{eq:m-bound}, since a graphon can be close to $q$ in cut norm while $\|W-q\|_2$ is of order one. However, we can use the fact that the linear term in \eqref{eq:coercive-1} is directly controlled by the cut norm.

\begin{proposition}\label{prop:barrier}
Fix $q\in(0,1)$ and set $b(q):=\tfrac18\kappa(q)$. Let $0<a\le b(q)$. If $W\in\cW$ satisfies $t(H,W)\le q^{e(H)}$ and $a\le\|W-q\|_\cut\le b(q)$, then
\begin{equation}\label{eq:barrier}
  \Ent(W)\ \ge\ h(q)+\min\Bigl\{\frac{a^2}{4},\ \frac{e(H)^2q^{2e(H)-2}}{64c_H^2\log^2(1/q)}\Bigr\}.
\end{equation}
\end{proposition}

\begin{proof}
Write $\sigma=\|W-q\|_2$ and $m=\dens(W)-q$. By \eqref{eq:norm-chain} we have $\sigma\ge\|W-q\|_\cut\ge a$, and taking $S=T=[0,1]$ in \eqref{eq:cut-norm} gives $|m|\le\|W-q\|_\cut\le b(q)$. Together with \eqref{eq:m-bound},
\begin{equation}\label{eq:two-bounds-on-m}
  m\le\min\Bigl\{b(q),\ \frac{c_H\sigma^3}{e(H)q^{e(H)-1}}\Bigr\}.
\end{equation}
Set $\sigma_\star:=e(H)q^{e(H)-1}/(2c_H\log(1/q))$, the threshold appearing in \eqref{eq:three-gaps}, and note that the choice $b(q)=\kappa(q)/8$ is precisely the one for which $\bigl(e(H)q^{e(H)-1}b(q)/c_H\bigr)^{1/3}=\tfrac12\sigma_\star$.

Suppose first that $\sigma\le\tfrac12\sigma_\star$. Then $c_H\sigma^3/(e(H)q^{e(H)-1})\le b(q)$, so \eqref{eq:coercive-1} and \eqref{eq:two-bounds-on-m} give
\[
  \Ent(W)-h(q)\ \ge\ \tfrac12\sigma^2-\frac{c_H\log(1/q)}{e(H)q^{e(H)-1}}\sigma^3
  =\sigma^2\Bigl(\tfrac12-\frac{c_H\log(1/q)}{e(H)q^{e(H)-1}}\sigma\Bigr)\ \ge\ \tfrac14\sigma^2\ \ge\ \tfrac14a^2,
\]
because $\bigl(c_H\log(1/q)/(e(H)q^{e(H)-1})\bigr)\cdot\tfrac12\sigma_\star=\tfrac14$.

Suppose instead that $\sigma>\tfrac12\sigma_\star$. Then \eqref{eq:coercive-1} and $m\le b(q)$ give
\[
  \Ent(W)-h(q)\ \ge\ -\log(1/q)\,b(q)+\tfrac12\Bigl(\tfrac12\sigma_\star\Bigr)^2
  =-\log(1/q)\,b(q)+\frac{e(H)^2q^{2e(H)-2}}{32c_H^2\log^2(1/q)}
  =\log(1/q)\,b(q),
\]
where the last step uses $e(H)^2q^{2e(H)-2}/\bigl(32c_H^2\log^2(1/q)\bigr)=2\log(1/q)\,b(q)$. 
\end{proof}

\Cref{prop:barrier} says that the constant graphon is separated from all other feasible configurations by a genuine entropy barrier, uniformly over compact subintervals of $(0,1)$ and with explicit constants. In particular, any path inside the feasible set which is continuous in the cut metric and joins $W\equiv q$ to a graphon of cost at most $h(q)$ must cross the annulus $a\le\|W-q\|_\cut\le b(q)$ and pay the price \eqref{eq:barrier}. Symmetry breaking therefore cannot occur through an infinitesimal bifurcation off the constant solution. This fact, in the form of \eqref{eq:three-gaps}, is what produces coexistence at $q_c$.

\section{A single threshold, coexistence, and a corner}\label{sec:transition}
In this section, we deduce our results about the nature of the phase transition from the local minimum estimates of the previous section. 

The following identity is essential, as it translates the increase in edge density to a strictly smaller entropy cost at smaller deficits. Restricted to two-valued graphons, the identity below is the computation Zhao used in \cite[\S4.2]{ZhaoLowerTail} to propagate symmetry breaking downward along the bipodal family. The new ingredient here is \cref{cor:gap}, which supplies a quantitative gap in the inequality $\dens(W)>q$ for an arbitrary competitor.

\begin{lemma}\label{lem:scaling}
For every $W\in\cW$, every $q\in[0,1]$ and every $s\in[0,1]$,
\begin{equation}\label{eq:scaling}
  \Ent(sW)-h(sq)=s\bigl(\Ent(W)-h(q)\bigr)+s\log s\,\bigl(\dens(W)-q\bigr),
\end{equation}
with the convention $s\log s=0$ at $s=0$, and $t(H,sW)=s^{e(H)}t(H,W)$.
\end{lemma}

\begin{proof}
From $h(x)=x\log x-x+1$ we get $\Ent(sW)=s\int W\log W+s\log s\,\dens(W)-s\,\dens(W)+1$ and, using $q\log q=h(q)+q-1$,
\[
  h(sq)=sq\log q+sq\log s-sq+1=s\,h(q)-s+sq\log s+1 .
\]
Since $\int W\log W=\Ent(W)+\dens(W)-1$, subtracting the two displays gives \eqref{eq:scaling}. The second assertion holds because $t(H,\cdot)$ is homogeneous of degree $e(H)$.
\end{proof}

\begin{corollary}\label{cor:downward}
Let $q\in(0,1)$ and suppose some $W$ satisfies \eqref{eq:competitor}. Then $\Phi(sq)<h(sq)$ for every $s\in(0,1)$.
\end{corollary}

\begin{proof}
By \cref{cor:gap} we have $\dens(W)>q$. In \eqref{eq:scaling} the first term on the right is nonpositive and the second is strictly negative, so $\Ent(sW)<h(sq)$; and $t(H,sW)=s^{e(H)}t(H,W)\le(sq)^{e(H)}$, so $sW$ is feasible at $sq$.
\end{proof}

\subsection{The critical parameter and coexistence}
The single critical parameter follows essentially immediately, as thinning any non-constant optimizer implies that the symmetry breaking regime must be an interval. Coexistence then follows, as the strict increase in edge density prevents non-constant optimizers from converging to the constant one at the critical parameter. 

By \cite[Theorem 2.7]{ZhaoLowerTail} the set $\{q\in(0,1):\Phi(q)<h(q)\}$ is nonempty and is not all of $(0,1)$. Define
\begin{equation}\label{eq:qc-def}
  q_c:=\sup\bigl\{q\in(0,1):\Phi(q)<h(q)\bigr\}\ \in\ (0,1).
\end{equation}

\begin{theorem}\label{thm:threshold}
The parameter $q_c$ defined in \eqref{eq:qc-def} satisfies all the conclusions of \cref{thm:variational}. Explicitly:
\begin{equation}\label{eq:phase-intervals}
  \Phi(q)<h(q)\ \ (0<q<q_c),\qquad \Phi(q)=h(q)\ \ (q_c\le q<1),
\end{equation}
the constant graphon is the unique minimizer for every $q>q_c$, and there exists a nonconstant $W_c\in\cM_{q_c}$ with
\begin{equation}\label{eq:coexistence}
  t(H,W_c)=q_c^{e(H)},\qquad \Ent(W_c)=h(q_c),\qquad \dens(W_c)\ge q_c+\kappa(q_c).
\end{equation}
\end{theorem}

\begin{proof}
We first show that the symmetry breaking set is an interval. Suppose $\Phi(q)<h(q)$ and let $W\in\cM_q$. By \cref{lem:constants} the graphon $W$ is nonconstant, and it satisfies \eqref{eq:competitor}, so \cref{cor:downward} gives $\Phi(q')<h(q')$ for every $q'<q$. Hence $\{q:\Phi(q)<h(q)\}$ is a downward closed subset of $(0,1)$ and therefore contains $(0,q_c)$. Its complement contains $(q_c,1)$ by the definition of the supremum, and $\Phi(q_c)=h(q_c)$ by continuity of $\Phi-h$ (\cref{lem:Phi-basic}). This is \eqref{eq:phase-intervals}, and it also shows that $q_c$ is the unique parameter with these properties.

We next prove uniqueness above $q_c$. Let $q>q_c$ and $W\in\cM_q$. If $W$ were nonconstant then, since $\Ent(W)=\Phi(q)=h(q)$, it would satisfy \eqref{eq:competitor}, and \cref{cor:downward} would give $\Phi(q')<h(q')$ for every $q'\in(q_c,q)$, contradicting \eqref{eq:phase-intervals}. So $W$ is constant, hence $W\equiv q$ by \cref{lem:constants}. This proves \cref{thm:variational}\ref{it:rs}.

To prove part \ref{it:sb}, note that if $q<q_c$ then $\Phi(q)<h(q)$, so every $W\in\cM_q$ is nonconstant by \cref{lem:constants} and satisfies \eqref{eq:competitor}, and we may apply \cref{cor:gap}. The same argument gives the last assertion of part \ref{it:coex}, since a nonconstant $W\in\cM_{q_c}$ has $\Ent(W)=\Phi(q_c)=h(q_c)$.

It remains to prove coexistence at $q_c$. The constant graphon lies in $\cM_{q_c}$ because $\Phi(q_c)=h(q_c)$. For the other phase, pick $q_k\uparrow q_c$ and $W_k\in\cM_{q_k}$. Each $W_k$ is nonconstant, so by \cref{cor:gap},
\begin{equation}\label{eq:uniform-gap-along-sequence}
  \dens(W_k)\ \ge\ q_k+\kappa(q_k).
\end{equation}
By compactness, pass to a subsequence with $W_k\to W_c$ in cut distance. Then $t(H,W_c)\le q_c^{e(H)}$, and \cref{lem:lsc} together with continuity of $\Phi$ gives
\[
  \Ent(W_c)\le\liminf_k\Ent(W_k)=\liminf_k\Phi(q_k)=\Phi(q_c),
\]
so $W_c\in\cM_{q_c}$. Since $\dens$ is continuous in the cut metric and $\kappa$ is continuous, \eqref{eq:uniform-gap-along-sequence} passes to the limit and yields $\dens(W_c)\ge q_c+\kappa(q_c)>q_c$; in particular $W_c$ is not constant, by \cref{lem:constants}. The equality $t(H,W_c)=q_c^{e(H)}$ is \cref{lem:Phi-basic}, and $\Ent(W_c)=\Phi(q_c)=h(q_c)$. This proves \eqref{eq:coexistence} and part \ref{it:coex}.
\end{proof}

\subsection{The corner in the rate function}
The right derivative of the rate function exists and can be explicitly computed since we know the value and solution of the variational problem. The difficulty comes from the left derivative, where we do not know any optimizers. Instead, we bound the left derivative by the derivative of an auxiliary rate function, formed by thinning the non-constant optimizer of maximal density at the critical parameter.  

Recall $\Phi(H,\eta)=\Phi(\eta^{1/e(H)})$, $\eta_c=q_c^{e(H)}$, and the latent jump $\Delta=\max_{W\in\cM_{q_c}}\dens(W)-q_c$ from \eqref{eq:Delta}. The maximum is attained by compactness of $\cM_{q_c}$ and continuity of $\dens$ in the cut metric and $\Delta\ge\kappa(q_c)$ is \eqref{eq:coexistence}. Let $W_c\in\cM_{q_c}$ attain it, so $\dens(W_c)=q_c+\Delta$.

The existence of the left derivative in \cref{thm:corner} follows from local semiconcavity of $\Phi$.
\begin{lemma}\label{lem:semiconcave}
Let $0<\alpha<\beta<1$. Then $q\mapsto\Phi(q)-q^2/(2\alpha^2)$ is concave on $[\alpha,\beta]$. In particular, $\Phi$ has finite one-sided derivatives at every point of $(0,1)$.
\end{lemma}

\begin{proof}
For $r\in[\alpha,\beta]$ choose $W_r\in\cM_r$ and put $g_r(q):=\Ent\bigl(\min\{(q/r)W_r,1\}\bigr)$. Since $t(H,\cdot)$ is monotone, $t\bigl(H,\min\{(q/r)W_r,1\}\bigr)\le(q/r)^{e(H)}t(H,W_r)\le q^{e(H)}$, so $\Phi\le g_r$ on $[\alpha,\beta]$, with equality at $q=r$. Hence
\[
  \Phi(q)=\min_{r\in[\alpha,\beta]}g_r(q),\qquad q\in[\alpha,\beta].
\]
Fix $w\in[0,1]$ and put $\phi(q):=h(\min\{qw/r,1\})$. Where $qw<r$ we have $\phi''(q)=(w/r)^2h''(qw/r)=w/(rq)\le1/q^2\le1/\alpha^2$, using $h''(x)=1/x$; where $qw>r$ the function $\phi$ is constant; and $\phi$ is continuously differentiable across $q=r/w$ because $h'(1)=0$. Thus $\phi(q)-q^2/(2\alpha^2)$ has a nonincreasing derivative and is concave on $[\alpha,\beta]$. Integrating over $w=W_r(x,y)$ shows that $g_r(q)-q^2/(2\alpha^2)$ is concave on $[\alpha,\beta]$, and hence so is the infimum of these functions over $r$. Finally, a concave function has finite one-sided derivatives at interior points of its interval of definition, and $\alpha,\beta$ were arbitrary.
\end{proof}

\begin{proof}[Proof of \cref{thm:corner}]
By \eqref{eq:phase-intervals}, $\Phi(H,\eta)=h(\eta^{1/e(H)})$ for $\eta\in[\eta_c,1)$, so
\[
  \Phi'_+(H,\eta_c)=\frac{\mathrm d}{\mathrm d\eta}\Big|_{\eta=\eta_c}h(\eta^{1/e(H)})=\frac{h'(q_c)}{e(H)q_c^{e(H)-1}}=\frac{\log q_c}{e(H)q_c^{e(H)-1}},
\]
which is \eqref{eq:right-derivative}.

Now define, for $q\in(0,q_c]$,
\begin{equation}\label{eq:C-def}
  C(q):=\Ent\Bigl(\tfrac q{q_c}W_c\Bigr)
  \overset{\eqref{eq:scaling}}{=}h(q)+\frac q{q_c}\log\frac q{q_c}\cdot\Delta,
\end{equation}
where we used $\Ent(W_c)=h(q_c)$ and $\dens(W_c)-q_c=\Delta$. Since $t\bigl(H,(q/q_c)W_c\bigr)=(q/q_c)^{e(H)}q_c^{e(H)}=q^{e(H)}$, the dilate $(q/q_c)W_c$ is feasible at $q$, so
\begin{equation}\label{eq:Phi-le-C}
  \Phi(q)\le C(q)\quad (0<q\le q_c),\qquad \Phi(q_c)=C(q_c)=h(q_c).
\end{equation}
Differentiating \eqref{eq:C-def} at $q_c$ from the left gives
\begin{equation*}
  C'_-(q_c)=h'(q_c)+\frac{\Delta}{q_c}=\log q_c+\frac{\Delta}{q_c}.
\end{equation*}
By \eqref{eq:Phi-le-C} and $q-q_c<0$,
\[
  \frac{\Phi(q)-\Phi(q_c)}{q-q_c}\ \ge\ \frac{C(q)-C(q_c)}{q-q_c}\qquad (q<q_c).
\]
By \cref{lem:semiconcave} the left derivative $\Phi'_-(q_c)$ exists, so letting $q\uparrow q_c$ gives
\begin{equation}\label{eq:Phi-left}
  \Phi'_-(q_c)\ \ge\ C'_-(q_c)=\log q_c+\frac{\Delta}{q_c}.
\end{equation}
Changing variables to $\eta=q^{e(H)}$, and using that $\eta\mapsto\eta^{1/e(H)}$ is differentiable at $\eta_c$ with derivative $1/(e(H)q_c^{e(H)-1})>0$, we see that $\Phi'_-(H,\eta_c)$ exists and that \eqref{eq:Phi-left} becomes
\[
  \Phi'_-(H,\eta_c)=\frac{\Phi'_-(q_c)}{e(H)q_c^{e(H)-1}}\ \ge\ \frac1{e(H)q_c^{e(H)-1}}\Bigl(\log q_c+\frac{\Delta}{q_c}\Bigr)
  =\Phi'_+(H,\eta_c)+\frac{\Delta}{e(H)q_c^{e(H)}},
\]
which is the first inequality in \eqref{eq:left-derivative}. The second follows from $\Delta\ge\kappa(q_c)$, since $\kappa(q_c)/(e(H)q_c^{e(H)})=e(H)q_c^{e(H)-2}/(8c_H^2\log^3(1/q_c))$. 
\end{proof}

\section{The conditioned random graph}\label{sec:conditioning}
In this section we prove our results about $G(n,p)$ conditioned on the lower tail event. Throughout this section $q\in(0,1)$ is fixed, $p=p_n\to0$ with $p_nn^{1/m_2(H)}\to\infty$ (so that also $np_n\to\infty$, since $m_2(H)>1$), and $\nu_{n,q}$ denotes the law of $G_{\cL}$, that is, of $G(n,p)$ conditioned on the lower tail event $\cL_p(H,q^{e(H)})$ of \eqref{eq:LT}. We write $Y=(Y_e)_{e\in \binom{[n]}2}$ for the edge-indicator vector under $\nu_{n,q}$, all expectations and probabilities being with respect to $\nu_{n,q}$ unless indicated otherwise. As before, each copy $A\in\cH(H)$ is identified with its edge set. For a bounded symmetric kernel $U$ we abbreviate
\begin{equation*}
  d_q(U):=\dist_\cut(U,\cM_q)=\inf_{W\in\cM_q}\delta_\cut(U,W).
\end{equation*}

\subsection{Pinning}
To start, we use the entropy increment scheme of \cite{KozmaSamotij} to select a small set of edges which decomposes the lower tail measure into a mixture of ``nearly product'' measures.
\begin{lemma}\label{lem:pinning}
For all $\alpha,\delta>0$ there is $L=L(\alpha,\delta)$ such that, whenever $p\ge Ln^{-1/m_2(H)}$ and $n$ is large, there is a deterministic $S\subseteq \binom{[n]}2$ with $|S|\le\alpha\binom n2$ and the following property. Define the (random) vector $\mathbf a=\mathbf a(Y_S)$ by
\begin{equation*}
  a_e:=\E\bigl[Y_e\mid Y_S\bigr]\ \ (e\notin S),\qquad a_e:=p\ \ (e\in S).
\end{equation*}
Then $0\le a_e\le p$ for every $e$, and with probability at least $1-\delta$,
\begin{equation}\label{eq:pinning-correlation}
  \sum_{A\cap S=\varnothing}\Bigl|\E[Y_A\mid Y_S]-\prod_{e\in A}a_e\Bigr|\ \le\ \delta\,\E[N_H(p)] ,
\end{equation}
and on the same event
\begin{equation}\label{eq:pinning-feasible}
  \sum_{A\in\cH(H)}\ \prod_{e\in A}a_e\ \le\ (q^{e(H)}+\delta+e(H)\alpha)\E[N_H(p)] .
\end{equation}
\end{lemma}

\begin{proof}
That $a_e\le p$ is \cref{thm:KS}\ref{it:KS-harris}. Apply \cref{thm:KS}\ref{it:KS-increment} to $\cH(H)$ with parameters $\alpha$ and $\beta$, obtaining $S$ with $|S|\le\alpha\binom n2$ and $\cE_H(S)\le\beta N_H$. Substituting into \eqref{eq:KS-CS} bounds the expectation of the left-hand side of \eqref{eq:pinning-correlation} by $\E[N_H(p)]\sqrt{(e(H)-1)\beta}$; choosing $\beta\le\delta^4/(e(H)-1)$ and applying Markov's inequality gives \eqref{eq:pinning-correlation} with probability at least $1-\delta$.

For \eqref{eq:pinning-feasible}, note that $N_H(Y)\le q^{e(H)}\E[N_H(p)]$ holds $\nu_{n,q}$-almost surely, so $\E[N_H(Y)\mid Y_S]\le q^{e(H)}\E[N_H(p)]$; combined with \eqref{eq:pinning-correlation} this controls the copies of $H$ disjoint from $S$. Every remaining copy meets $S$, and since each edge of $K_n$ lies in exactly $e(H)N_H/\binom n2$ copies of $H$, there are at most $e(H)|S|N_H/\binom n2\le e(H)\alpha N_H$ of these; each contributes at most $p^{e(H)}$ since $a_e\le p$.
\end{proof}

\subsection{Stability around the whole minimizer set}
Next, we show that the minimizer set of the variational problem is stable, in the sense that a near optimizer must be close in cut distance to the set of true minimizers.
\begin{lemma}\label{lem:graphon-gap}
For every $q\in(0,1)$ and every $\varepsilon>0$ there are $\theta_0,\gamma>0$ such that
\begin{equation}\label{eq:graphon-gap}
  t(H,W)\le q^{e(H)}+\theta_0\ \text{ and }\ d_q(W)\ge\varepsilon
  \quad\Longrightarrow\quad
  \Ent(W)\ge\Phi(q)+2\gamma .
\end{equation}
The implication persists, with the same $\gamma$, when $\theta_0$ is replaced by any smaller positive number.
\end{lemma}

\begin{proof}
If not, there are $\theta_k\downarrow0$ and graphons $W_k$ with $t(H,W_k)\le q^{e(H)}+\theta_k$, $d_q(W_k)\ge\varepsilon$ and $\Ent(W_k)\le\Phi(q)+1/k$. Passing to a cut convergent subsequence with limit $W$, we get $t(H,W)\le q^{e(H)}$ and, by \cref{lem:lsc}, $\Ent(W)\le\Phi(q)$; hence $W\in\cM_q$, contradicting $d_q(W)\ge\varepsilon$, which passes to the limit because $d_q$ is $1$-Lipschitz. The last sentence is immediate, since shrinking $\theta_0$ strengthens the hypothesis.
\end{proof}

Fix $\varepsilon>0$ and let $\theta_0,\gamma$ be as in \cref{lem:graphon-gap}. For $0<\theta\le\theta_0$ with $q^{e(H)}+\theta<1$ put
\begin{equation}\label{eq:Lnp}
  L_{n,p}(\theta):=\Phi_p(H,q^{e(H)}+\theta),
\end{equation}
\begin{multline*}
  K_{n,p}(\theta,\varepsilon):=\inf\Bigl\{\sum_e\ip(x_e)\ :\ \mathbf x\in[0,p]^{\binom{[n]}2},
  \E[N_H(\mathbf x)]\le(q^{e(H)}+\theta)\E[N_H(p)],\
  d_q\bigl(W_{\mathbf x/p}\bigr)\ge\varepsilon\Bigr\}.
\end{multline*}

\begin{lemma}\label{lem:finite-gap}
With the notation above, $K_{n,p_n}(\theta,\varepsilon)-L_{n,p_n}(\theta)\ge\gamma\binom n2p_n$ for all sufficiently large $n$.
\end{lemma}

\begin{proof}
Apply \eqref{eq:restricted-liminf} with $\cK=\cM_q$, which is compact, and $\eta_n\equiv q^{e(H)}+\theta$; combined with \eqref{eq:graphon-gap} it gives
$\liminf_nK_{n,p_n}(\theta,\varepsilon)/(\binom n2p_n)\ge\Phi(q)+2\gamma$. On the other hand \eqref{eq:finite-limit} gives $L_{n,p_n}(\theta)/(\binom n2p_n)\to \Phi(H,q^{e(H)}+\theta)\le \Phi(H,q^{e(H)})=\Phi(q)$.
\end{proof}

\subsection{Localization of the conditional marginals}
Combining the entropy decomposition and stability allows us to show that on most components of the decomposed measure, the individual edge marginals must lie close to the minimizer set. Let
\begin{equation*}
  D_n:=\KL\bigl(\nu_{n,q}\,\big\|\,G(n,p)\bigr)=-\log\Pr\bigl(\cL_p(H,q^{e(H)})\bigr),
\end{equation*}
the identity being the standard formula for the relative entropy of a conditioned measure. Fix a pinning set $S$. For $s$ in the support of $Y_S$, let $\nu_s$ be the conditional law of $(Y_e)_{e\notin S}$ and
\begin{equation*}
  Q_s:=\bigotimes_{e\notin S}\Ber\bigl(a_e(s)\bigr)
\end{equation*}
the product measure with the same one-edge marginals. The next lemma decomposes the entropy cost of the selected component into the remaining correlations and the entropy cost of the individual marginals.

\begin{lemma}\label{lem:projection}
For every such $s$,
\begin{equation}\label{eq:projection}
  \KL\bigl(\nu_s\,\big\|\,\Ber(p)^{\otimes(E\setminus S)}\bigr)=\KL(\nu_s\|Q_s)+\sum_{e\notin S}\ip\bigl(a_e(s)\bigr),
\end{equation}
and consequently
\begin{equation}\label{eq:two-consequences}
  \E\Bigl[\sum_{e}\ip(a_e)\Bigr]\le D_n
  \qquad\text{and}\qquad
  \E_{Y_S}\KL\bigl(\nu_{Y_S}\|Q_{Y_S}\bigr)\le D_n-\E\Bigl[\sum_e\ip(a_e)\Bigr].
\end{equation}
\end{lemma}

\begin{proof}
Writing $\mathrm d\nu_s/\mathrm dP=(\mathrm d\nu_s/\mathrm dQ_s)\cdot(\mathrm dQ_s/\mathrm dP)$ with $P=\Ber(p)^{\otimes(E\setminus S)}$ and taking $\nu_s$-expectations, the second factor contributes
\[
  \E_{\nu_s}\sum_{e\notin S}\Bigl[Y_e\log\frac{a_e}{p}+(1-Y_e)\log\frac{1-a_e}{1-p}\Bigr]=\sum_{e\notin S}\ip(a_e),
\]
because $\E_{\nu_s}Y_e=a_e$. That $\nu_s\ll Q_s$ holds since $a_e=0$ forces $Y_e=0$ under $\nu_s$, and $a_e\le p<1$ always. This is \eqref{eq:projection}. Averaging over $s$ and using the chain rule
\[
  D_n=\KL\bigl(\operatorname{Law}_{\nu_{n,q}}(Y_S)\,\big\|\,\Ber(p)^{\otimes S}\bigr)
  +\E_{Y_S}\KL(\nu_{Y_S}\|P)\ \ge\ \E_{Y_S}\KL(\nu_{Y_S}\|P)
\]
gives both statements in \eqref{eq:two-consequences}, on noting that $\ip(p)=0$ so that $\sum_{e\notin S}\ip(a_e)=\sum_e\ip(a_e)$.
\end{proof}

Since a component with marginals far from the minimizer set has a strictly greater entropy cost, the lower tail measure cannot afford to have many such components. This is formalized in the next proposition, showing that with large probability the marginals are close to the minimizer set. 
\begin{proposition}\label{prop:localization}
Fix $q\in(0,1)$ and $\varepsilon,\zeta>0$, and let $\alpha_0>0$ be prescribed. There is $\alpha\in(0,\alpha_0]$ such that, for all sufficiently large $n$, the pinning set $S=S_n$ of \cref{lem:pinning} satisfies $|S|\le\alpha\binom n2$ and
\begin{equation}\label{eq:localization}
  \nu_{n,q}\bigl(d_q(W_{\mathbf a/p})>\varepsilon\bigr)<\zeta
  \qquad\text{and}\qquad
  \E_{Y_S}\KL\bigl(\nu_{Y_S}\|Q_{Y_S}\bigr)<\zeta n^2p .
\end{equation}
\end{proposition}

\begin{proof}
Let $\theta_0,\gamma>0$ be given by \cref{lem:graphon-gap} for the parameter $\varepsilon$. Choose $0<\theta\le\theta_0$ with $q^{e(H)}+\theta<1$, then $\delta>0$ and $\alpha\in(0,\alpha_0]$, subject to
\begin{equation}\label{eq:parameter-choice}
  \delta+e(H)\alpha<\theta ,
\end{equation}
the precise smallness being fixed at the end of the proof. Apply \cref{lem:pinning} with these $\alpha,\delta$ and let $\cG$ be the good event there, so $\Pr(\cG)\ge1-\delta$.

On $\cG$, the vector $\mathbf a$ satisfies $0\le a_e\le p$ and, by \eqref{eq:pinning-feasible} and \eqref{eq:parameter-choice}, the constraint of \eqref{eq:Lnp}; hence $\sum_e\ip(a_e)\ge L_{n,p}(\theta)$. If in addition $d_q(W_{\mathbf a/p})\ge\varepsilon$ then $\sum_e\ip(a_e)\ge K_{n,p}(\theta,\varepsilon)$. Since $\ip\ge0$, discarding $\cG^c$ gives
\begin{equation}\label{eq:mixture}
  \E\sum_e\ip(a_e)\ \ge\ \Pr(\cG)L_{n,p}(\theta)+\Pr\bigl(\cG\cap\{d_q(W_{\mathbf a/p})\ge\varepsilon\}\bigr)\bigl(K_{n,p}(\theta,\varepsilon)-L_{n,p}(\theta)\bigr).
\end{equation}
Divide by $\binom n2p$ and use \eqref{eq:two-consequences}, \cref{cor:ldp}, \cref{prop:finite-limit} and \cref{lem:finite-gap}:
\begin{equation}\label{eq:localization-bound}
  \limsup_{n\to\infty}\ \Pr\bigl(\cG\cap\{d_q(W_{\mathbf a/p})\ge\varepsilon\}\bigr)\ \le\ \frac{\Phi(q)-(1-\delta)\Phi(H,q^{e(H)}+\theta)}{\gamma}.
\end{equation}
Since $\Phi(H,\cdot)$ is continuous and $\Phi(H,q^{e(H)})=\Phi(q)$, and since $\gamma$ does not depend on $\theta$ or $\delta$, the right-hand side of \eqref{eq:localization-bound} is at most $\zeta/2$ once $\theta$ and $\delta$ are chosen small enough; shrinking $\delta$ further so that $\delta<\zeta/2$ and adding $\Pr(\cG^c)\le\delta$ proves the first half of \eqref{eq:localization}.

For the second half, combine the second inequality of \eqref{eq:two-consequences} with the bound $\E\sum_e\ip(a_e)\ge(1-\delta)L_{n,p}(\theta)$ from \eqref{eq:mixture}:
\[
  \E_{Y_S}\KL(\nu_{Y_S}\|Q_{Y_S})\ \le\ D_n-(1-\delta)L_{n,p}(\theta).
\]
Dividing by $n^2p$ and using $\binom n2p/(n^2p)\to\tfrac12$ together with \cref{cor:ldp} and \eqref{eq:finite-limit},
\[
  \limsup_{n\to\infty}\frac1{n^2p}\E_{Y_S}\KL(\nu_{Y_S}\|Q_{Y_S})\ \le\ \tfrac12\bigl(\Phi(q)-(1-\delta)\Phi(H,q^{e(H)}+\theta)\bigr),
\]
which is less than $\zeta$ for $\theta,\delta$ small.
\end{proof}

\subsection{From the marginals to the sampled graph}
Finally, we transfer closeness in marginals to the conditional measure through the property that each component only has weak correlations among the edge marginals. This allows us to conclude that the full conditioned graph must be close to the optimizer set of the variational problem. 

For $\mathbf x\in[0,\infty)^{\binom{[n]}2}$ recall that $A_{\mathbf x}$ is the associated $n$-part step kernel, so $A_{\mathbf x}/p=W_{\mathbf x/p}$; for $R\subseteq \binom{[n]}2$, $A_{\mathbf x}\mathbf 1_R$ denotes the kernel obtained by keeping the two blocks of each edge of $R$ and zeroing all others. We need two standard estimates.

\begin{lemma}\label{lem:product-cut}
There is an absolute constant $c>0$ with the following property. Let $(X_e)_{e\in \binom{[n]}2}$ be independent with $X_e\sim\Ber(a_e)$ and $a_e\le p$, where $np\to\infty$. Then for every $\varepsilon\in(0,1)$ and all large $n$,
\begin{equation}\label{eq:product-cut}
  \Pr\bigl(\|A_X-A_{\mathbf a}\|_\cut>\varepsilon p\bigr)\ \le\ \exp\bigl(-c\,\varepsilon^2n^2p\bigr).
\end{equation}
The same bound holds if only the edges of a prescribed subset are sampled and the remaining entries are set to zero.
\end{lemma}

\begin{proof}
For a kernel that is constant on the blocks $I_i\times I_j$, the supremum in \eqref{eq:cut-norm} may be restricted to unions of the intervals $I_i$: writing a measurable set through its intersection proportions with the $I_i$, the functional is affine in each proportion, so an extreme point may be chosen. Fix $S,T\subseteq[n]$. Then
\begin{gather*}
  \int_{S\times T}(A_X-A_{\mathbf a})=\frac1{n^2}\sum_{e=\{i,j\}}c_e\,(X_e-a_e),\\
  c_e:=\mathbf 1[i\in S,j\in T]+\mathbf 1[j\in S,i\in T]\in\{0,1,2\}.
\end{gather*}
The summands are independent, centered, bounded by $2$ in absolute value, and have total variance at most $4\sum_ea_e\le2n^2p$. Bernstein's inequality gives
\[
  \Pr\Bigl(\Bigl|\sum_ec_e(X_e-a_e)\Bigr|>\varepsilon pn^2\Bigr)\le2\exp\bigl(-\varepsilon^2n^2p/6\bigr).
\]
There are at most $4^n$ pairs $(S,T)$, and $4^n=\exp(o(\varepsilon^2n^2p))$ because $np\to\infty$; a union bound therefore gives \eqref{eq:product-cut} with $c=1/12$ for large $n$. Restricting to a subset of edges only decreases the variance.
\end{proof}

\begin{lemma}\label{lem:transfer}
Let $Q\ll P$ be probability measures on a common measurable space and let $B$ be an event with $P(B)\le e^{-L}$ for some $L>0$. Then $Q(B)\le\bigl(\KL(Q\|P)+\log2\bigr)/L$.
\end{lemma}

\begin{proof}
By the data-processing inequality applied to the two-point partition $\{B,B^c\}$,
\[
  \KL(Q\|P)\ \ge\ Q(B)\log\frac{Q(B)}{P(B)}+\bigl(1-Q(B)\bigr)\log\frac{1-Q(B)}{1-P(B)}
  \ \ge\ Q(B)L-\log2,
\]
where we bounded
\[
  Q(B)\log Q(B)+\bigl(1-Q(B)\bigr)\log\bigl(1-Q(B)\bigr)\ \ge\ -\log2
\]
and used $\log(1/P(B))\ge L$ and $1-P(B)\le1$.
\end{proof}

The first estimate shows that graphs sampled from independent marginals are close to their associated graphons, and the second transfers this statement to joint distributions with only weak correlations. Applying these two estimates to the decomposed components of the conditional measure proves \cref{thm:probabilistic}\ref{it:cut}.

\begin{proof}[Proof of \cref{thm:probabilistic}\ref{it:cut}]
Fix $\varepsilon>0$ and $\xi\in(0,1)$, and let $c$ be the constant of \cref{lem:product-cut}. Apply \cref{prop:localization} with parameters $\varepsilon/3$ and $\zeta:=\xi$, and with $\alpha_0:=\xi$, to obtain a pinning set $S$ with $|S|\le\alpha\binom n2$, $\alpha\le\xi$, such that for all large $n$
\begin{equation}\label{eq:two-inputs}
  \nu_{n,q}\bigl(d_q(W_{\mathbf a/p})>\varepsilon/3\bigr)<\xi,
  \qquad
  \E_{Y_S}\KL(\nu_{Y_S}\|Q_{Y_S})<\xi n^2p .
\end{equation}
Decompose
\begin{equation*}
  \dist_\cut\Bigl(\frac{A_Y}p,\ \cM_q\Bigr)\ \le\ d_q\bigl(W_{\mathbf a/p}\bigr)
  +\Bigl\|\frac{(A_Y-A_{\mathbf a})\mathbf 1_{E\setminus S}}p\Bigr\|_\cut
  +\Bigl\|\frac{(A_Y-A_{\mathbf a})\mathbf 1_{S}}p\Bigr\|_\cut .
\end{equation*}
The first term exceeds $\varepsilon/3$ with probability less than $\xi$ by \eqref{eq:two-inputs}.

For the second term, condition on $Y_S=s$. Under $Q_s$ the unpinned coordinates are independent with means $a_e\le p$, so \cref{lem:product-cut} gives
\[
  Q_s\Bigl(\Bigl\|\frac{(A_Y-A_{\mathbf a})\mathbf 1_{E\setminus S}}p\Bigr\|_\cut>\frac\varepsilon3\Bigr)\le\exp\bigl(-c'\varepsilon^2n^2p\bigr),\qquad c':=c/9,
\]
uniformly over $s$ in the support of $Y_S$ and for all large $n$. Since $\nu_s\ll Q_s$ (\cref{lem:projection}), \cref{lem:transfer} with $L=c'\varepsilon^2n^2p$ and averaging over $s$ give
\begin{equation*}
  \nu_{n,q}\Bigl(\Bigl\|\frac{(A_Y-A_{\mathbf a})\mathbf 1_{E\setminus S}}p\Bigr\|_\cut>\frac\varepsilon3\Bigr)
  \le\frac{\E_{Y_S}\KL(\nu_{Y_S}\|Q_{Y_S})+\log2}{c'\varepsilon^2n^2p}
  \le\frac{\xi}{c'\varepsilon^2}+o(1),
\end{equation*}
using \eqref{eq:two-inputs} and $n^2p\to\infty$.

For the third term, recall $a_e=p$ for $e\in S$ and $\E_{\nu_{n,q}}Y_e\le p$ by \cref{thm:KS}\ref{it:KS-harris}. Each edge occupies two blocks of area $n^{-2}$, so
\begin{equation*}
  \E_{\nu_{n,q}}\Bigl\|\frac{(A_Y-A_{\mathbf a})\mathbf 1_S}p\Bigr\|_\cut
  \le\frac2{n^2p}\sum_{e\in S}\E|Y_e-p|\le\frac{4|S|}{n^2}\le2\alpha\le2\xi ,
\end{equation*}
so by Markov's inequality this term exceeds $\varepsilon/3$ with probability at most $6\xi/\varepsilon$.

Combining the three estimates,
\[
  \limsup_{n\to\infty}\ \nu_{n,q}\Bigl(\dist_\cut\bigl(A_Y/p,\cM_q\bigr)>\varepsilon\Bigr)\ \le\ \xi+\frac{\xi}{c'\varepsilon^2}+\frac{6\xi}\varepsilon .
\]
As $\xi>0$ was arbitrary and $\varepsilon$ is fixed, the left-hand side is $0$.
\end{proof}

\subsection{Consequences for the conditioned graph}\label{sec:consequences}

We first show that the number of copies of $H$ is the maximum allowed by the lower tail constraint, proving \cref{thm:probabilistic}\ref{it:sat}.

\begin{proof}[Proof of \cref{thm:probabilistic}\ref{it:sat}]
The upper bound holds almost surely. For the lower bound fix $\delta\in(0,q^{e(H)})$ and set $q_-:=(q^{e(H)}-\delta)^{1/e(H)}<q$. Since $\cL_p(H,q_-^{e(H)})\subseteq\cL_p(H,q^{e(H)})$,
\[
  \Pr\Bigl(N_H(G_{\cL})\le(q^{e(H)}-\delta)\E[N_H(p_n)]\Bigr)=\frac{\Pr(\cL_p(H,q_-^{e(H)}))}{\Pr(\cL_p(H,q^{e(H)}))} .
\]
By \cref{cor:ldp} the logarithm of this ratio, divided by $\binom n2p_n$, tends to $\Phi(q)-\Phi(q_-)$, which is strictly negative because $\Phi$ is strictly decreasing (\cref{lem:Phi-basic}). Since $n^2p_n\to\infty$, the ratio tends to $0$.
\end{proof}

We conclude by showing that the edge density is an order parameter for the phase transition, proving \cref{thm:probabilistic}\ref{it:order}. Let
\begin{equation*}
  R_q:=\bigl\{\dens(W):W\in\cM_q\bigr\}\subseteq[0,1],
\end{equation*}
a compact set, being the continuous image of the compact set $\cM_q$.

\begin{lemma}\label{cor:edge-localization}
Under the hypotheses of \cref{thm:probabilistic}, $\dist\bigl(e(G_{\cL})/(\binom n2p_n),\,R_q\bigr)\to0$ in probability under $\nu_{n,q}$.
\end{lemma}

\begin{proof}
The integral of a kernel is unchanged by relabeling and is dominated by the cut norm, so for every $W$ we have $|\dens(A_{G_{\cL}}/p_n)-\dens(W)|\le\delta_\cut(A_{G_{\cL}}/p_n,W)$; hence
\[
  \dist\bigl(\dens(A_{G_{\cL}}/p_n),\,R_q\bigr)\ \le\ \dist_\cut\bigl(A_{G_{\cL}}/p_n,\,\cM_q\bigr),
\]
which tends to $0$ by \cref{thm:probabilistic}\ref{it:cut}. It remains to replace $\dens(A_{G_{\cL}}/p_n)$ by the normalized edge count. Since the diagonal blocks vanish,
\[
  \dens\Bigl(\frac{A_{G_{\cL}}}{p_n}\Bigr)=\frac{2e(G_{\cL})}{n^2p_n}=\Bigl(1-\frac1n\Bigr)\frac{e(G_{\cL})}{\binom n2p_n},
\]
so the two quantities differ by $\bigl(e(G_{\cL})/\binom n2p_n\bigr)/n\le1/(np_n)\to0$.
\end{proof}

\begin{proof}[Proof of \cref{thm:probabilistic}\ref{it:order}]
By \cref{thm:threshold}, $R_q=\{q\}$ when $q>q_c$; $R_q\subseteq[q+\kappa(q),1]$ when $q<q_c$; and $R_{q_c}\subseteq\{q_c\}\cup[q_c+\kappa(q_c),\,q_c+\Delta]$. In each case the assertion follows from \cref{cor:edge-localization}.
\end{proof}

\end{document}